\documentclass[12pt, a4paper]{article}
\usepackage[utf8]{inputenc}
\usepackage{amsmath,amssymb,amsthm}
\usepackage[inline]{enumitem}
\usepackage[nocompress, noadjust]{cite}
\usepackage{color,hyperref}
\usepackage{tikz}
\usetikzlibrary{calc}
\usepackage{url}
\usepackage{authblk}
\usepackage{booktabs}
\usepackage{makecell}
\usepackage{mleftright}
\mleftright

\newtheorem{theorem}{Theorem}

\newtheorem{lemma}[theorem]{Lemma}

\newtheorem{corollary}[theorem]{Corollary}

\newtheorem{conjecture}[theorem]{Conjecture}
\newtheorem{problem}[theorem]{Problem}
\newtheorem{construction}{Construction}

\title{Perfect shuffling with fewer lazy transpositions}
\author{Carla Groenland\footnote{Utrecht University, Utrecht, The Netherlands, \href{mailto:c.e.groenland@uu.nl}{c.e.groenland@uu.nl}. Partially supported by the ERC Horizon 2020 project CRACKNP (grant no. 853234).}
\quad
Tom Johnston\footnote{School of Mathematics, University of Bristol, Bristol, BS8 1UG, UK and Heilbronn Institute for Mathematical Research, Bristol, UK, \href{mailto:tom.johnston@bristol.ac.uk}{tom.johnston@bristol.ac.uk}.}\\
Jamie Radcliffe\footnote{University of Nebraska-Lincoln, USA,\href{mailto:jamie.radcliffe@unl.edu}{jamie.radcliffe@unl.edu}.}
\quad
Alex Scott\protect\footnotemark[2] 
\footnote{Mathematical Institute, University of Oxford, Oxford OX2 6GG, UK, \href{mailto:scott@maths.ox.ac.uk}{scott@maths.ox.ac.uk}. Supported by EPSRC grant EP/V007327/1.}}

\date{\today}

\newcommand{\expec}[1]{\mathbb{E}\left[#1\right]}
\newcommand{\floor}[1]{\left \lfloor #1 \right \rfloor}
\DeclareMathOperator{\unif}{Uniform}
\DeclareMathOperator{\Image}{Im}

\newcommand{\eps}{\varepsilon}

\tikzstyle{vertex}=[circle, draw, fill, inner sep=0pt, minimum width=4pt]

\begin{document}
\maketitle
\begin{abstract}
A lazy transposition $(a,b,p)$ is the random permutation that equals the identity with probability $1-p$ and the transposition $(a,b)\in S_n$ with probability $p$. 
How long must a sequence of independent lazy transpositions be if their composition is uniformly distributed? 
It is known that there are sequences of length $\binom{n}2$, but are there shorter sequences?  This was raised by Fitzsimons in 2011, and independently by Angel and Holroyd in 2018. 
We answer this question negatively by giving a construction of length $\frac23 \binom{n}2+O(n\log n)$, and consider some related questions.
\end{abstract}

\section{Introduction}

Let $S_n$ be the symmetric group on $n$ elements, and write $(a,b)\in S_n$ for the transposition that swaps $a$ and $b$, and $1\in S_n$ for the identity permutation.
A \emph{lazy transposition} $T=(a, b, p)$ is the random permutation
\[
T=\begin{cases}
	(a,b) & \text{ with probability }p, \\
	1     & \text{ otherwise}.
\end{cases}
\]
The composition of a sequence of lazy transpositions is also a random permutation, and we will be interested in sequences which generate a uniformly random permutation. Let $U(n)$ denote the minimum length $\ell$ for which there exists a sequence of independent lazy transpositions $T_1, \dots, T_\ell$ such that $T_1\cdots T_\ell\sim \unif(S_n)$, the uniform distribution over $S_n$. We will call such a sequence a \emph{transposition shuffle}. These were first discussed on Stack Exchange in 2011 \cite{stack}, before being investigated in more depth by Angel and Holroyd in 2018 \cite{angel2018perfect}. Both give constructions using $\binom{n}{2}$ lazy transpositions and ask if this is best possible.

\begin{problem}[Fitzsimons \cite{stack}, Angel and Holroyd \cite{angel2018perfect}]
	\label{prob:main}
	Does $U(n)=\binom{n}2$ for all $n$?
\end{problem}
The main result of this paper is a new construction for shuffling with lazy transpositions which improves the constant in the upper bound of $U(n)$ and answers Problem \ref{prob:main} in the negative.

\begin{theorem}
	\label{thm:wrongconstant}
	There exists an $\eps > 0$ such that $U(n) \leq (1 - \eps) \binom{n}{2}$ for all $n \geq 6$. Moreover, 
	\[ U(n) \leq \frac{2}{3} \binom{n}{2} + O(n \log n).\]
\end{theorem}
It will often be convenient to think of transpositions as moving $n$ distinguishable ``counters" across $n$ ``positions" which we label $1, \dots, n$. A transposition $(a,b)$ acts by switching the counters currently in positions $a$ and $b$, so a lazy transposition $(a,b,p)$ switches the counters currently in positions $a$ and $b$ with probability $p$, and does nothing with probability $1-p$. In this language a transposition shuffle is a sequence of independent lazy transpositions such that the final order of the counters is uniformly distributed over all $n!$ orders.

Transposition shuffles are a special case of the more general problem of shuffling $t$ counters across $n$ positions. We start with counters in positions $1, \dots, t$, and leave the other positions empty. A transposition $(a,b)$ switches any counters in positions $a$ and $b$ (which may include switching a counter with an empty space). A sequence of lazy transpositions naturally defines a random variable by the positions of the counters at the end of sequence, and we can again ask that this be uniformly distributed over all $t! \binom{n}{t}$ options. If the output is uniformly distributed, we say the sequence achieves \emph{$t$-uniformity on $n$ points} and call the sequence a \emph{$(t,n)$-shuffle} (so a transposition shuffle is an $(n,n)$-shuffle). Let $U_t(n)$ denote the minimum number of lazy transpositions in a $(t,n)$-shuffle. The special case of shuffling all elements is $U_n(n) = U(n)$. The constructions in Section \ref{sec:preliminaries} give an upper bound $U_t(n) \leq tn - \binom{t+1}{2}$ and, in light of Problem \ref{prob:main}, it is natural to ask the following. 

\begin{problem}
	\label{prob:tn}
	Does $U_t(n) = tn - \binom{t+1}{2}$?
\end{problem}

When $t = n$, this is equivalent to Problem \ref{prob:main}, and Theorem \ref{thm:wrongconstant} shows that the bound can be improved by a constant factor. But what about other values of $t$? The constructions used in the proof of Theorem \ref{thm:wrongconstant} give $(3,n)$- and $(4,n)$-shuffles which beat the upper bound by a constant factor, and we obtain the following theorem.

\begin{theorem}
	\label{thm:kn}
	There exists a constant $\eps > 0$ such that \[U_t(n) \leq (1 - \eps ) \left(tn - \binom{t+1}{2} \right) \] for all $3 \leq t \leq n$ and $n \geq 6$. Moreover, for a fixed $t = 3 \ell + r$, we have 
	\[
	U_t(n) \leq (2\ell + r)n + O_\ell\left(\log n\right).
	\]
\end{theorem}

We remark that it is important that the output permutation is exactly uniform; it was shown by Czumaj \cite{czumaj2015random} that there are sequences of lazy transpositions of length $O(n \log n)$ which are very close to uniform. 

\begin{theorem}[Czumaj \cite{czumaj2015random}]
	Let $c$ be an arbitrary constant. There is a sequence of $O(n \log n)$ independent lazy transpositions $T_1, \dots, T_\ell$, all with probability $1/2$, such that the total variation distance between the distribution of $T_1 \cdots T_\ell$ and $\unif(S_n)$ is $O(n^{-c})$.
\end{theorem}

It is not hard to see that any transposition shuffle must have length $\Omega(n \log n)$. Indeed, suppose $T_1, \dots, T_\ell$ is a transposition shuffle, so that the composition $T_1 \cdots T_\ell$ is uniformly distributed over $S_n$ and supported on $n!$ permutations. Each lazy transposition can at most double the size of the support, and so $T_1 \cdots T_\ell$ is supported on at most $2^\ell$ permutations. Hence, $\ell \geq \log_2(n!) = \Theta(n \log n)$. Despite the fact that this argument relies only on reaching every state with positive probability, it gives the best known lower bound for the length of a transposition shuffle.

It is also interesting to consider transposition shuffles in which the transpositions are restricted to a special set.
Angel and Holroyd \cite{angel2018perfect} studied the problem of constructing transposition shuffles when all transpositions are of the form $(i, i+1)$, and they classified the minimum transposition shuffles using reduced words.
It is easy to see that any transposition shuffle using transpositions of the form $(i, i+1)$ requires $\binom{n}{2}$ transpositions just to give positive probability to the reverse permutation, and they showed that transposition shuffles of this length can indeed be achieved.

Here, we consider the problem of transposition shuffles which only use transpositions of the form $(1, \cdot)$, which we call \emph{star transpositions} (as they match the edges of a star graph).
This case is similar to the set-up of Problem \ref{prob:main}: a minimum star transposition shuffle must use $\Theta(U(n))$ transpositions. Indeed, by replacing the general transposition $(i,j,p)$ by the star transpositions $(1,i,1), (1,j,p), (1,i,1)$, any transposition shuffle can easily be converted to one using star transpositions with only triple the number of transpositions. In particular, showing that the minimum number of transpositions in this highly restrictive model is $\Omega(n^2)$ would show that the general case is $\Omega(n^2)$ as well. 

One might hope that the analogue of Problem \ref{prob:main} is true when restricted to star transpositions: if all transpositions are of the form $(1, \cdot)$, does every transposition shuffle have at least $\binom{n}{2}$ transpositions? The following theorem shows that this is unfortunately not the case.

\begin{theorem}
\label{thm:improved-star}
Let $t = 4 \ell + r$ for some $0 \leq r < 3$. Let $\tilde{U}_t(n)$ be the minimum length of a $(t,n)$-shuffle in which every transposition is a star transposition. Then
\[\tilde{U}_t(n) \leq \left( \frac{7\ell}{2} + r\right)n + O_\ell\left( \log n\right).\]
Moreover, 
\[\tilde{U}(n) \leq \frac{7}{8} \binom{n}{2} + O\left(n \log n\right).\]
\end{theorem}

The rest of the paper is organised as follows.
In Section \ref{sec:preliminaries} we recall some existing constructions which achieve the trivial bound and will be useful throughout the rest of the paper.
We give our new construction (Construction \ref{constr:main}) in Section \ref{sec:un} and deduce Theorem \ref{thm:wrongconstant} and Theorem \ref{thm:kn}. In Section \ref{sec:improved-star} we study star transpositions, showing that the answer to Problem \ref{prob:main} is negative even in this special case (Theorem \ref{thm:improved-star}). We finish with some open problems in Section \ref{sec:discussion}.

\section{Preliminaries}
\label{sec:preliminaries}
In this section we give simple ``sweeping" and ``divide and conquer" constructions for $(t,n)$-shuffles, all of which give upper bounds of $U_t(n) \leq tn - \binom{t+1}{2}$. Combining these constructions gives the following lemma mentioned by Angel and Holroyd \cite{angel2018perfect} which we prove at the end of this section.

\begin{lemma}[\cite{angel2018perfect}]
	\label{lem:constant}
	Suppose that $U(n_0) \leq (1-\eps)\binom{n_0}{2}$. Then there exists $\delta > 0$ such that 
	\[U(n) \leq (1 - \delta) \binom{n}{2}\] for all $n \geq n_0$. 
	Explicitly, we may take \[\delta = \frac{n_0 -1}{4 n_0} \eps.\]
\end{lemma}

We start by giving a simple sweeping construction which extends a $(t, n -1)$-shuffle to a $(t,n)$-shuffle by adding $t$ transpositions. Note that there are other variations of the sweeping construction, and the variation here was chosen as all transpositions are star transpositions (provided an appropriate $(t, n-1)$-shuffle is chosen). 

\begin{lemma}[\cite{stack, angel2018perfect}]
	\label{lem:sweeping}
	Let $n \geq t + 1$. Then
	\[U_{t}(n) \leq U_{t}(n - 1) + t.\]
	
\end{lemma}
\begin{proof}[Sketch proof]
	The following gives a $(t, n)$-shuffle.
	\begin{enumerate}
		\item Start with the transpositions $(1,2, 1/2), (1,3, 1/3), \dots, (1,t, 1/t)$ and the transposition $(1,n, t/n)$.
		\item Append a $(t, n-1)$-shuffle. \qedhere
	\end{enumerate}
\end{proof}

Clearly, shuffling $n-1$ counters over $n$ positions is the same as shuffling all $n$ counters, so this construction gives the bound
\begin{equation}
\label{eqn:sweeping-n}
U(n) = U_{n-1}(n) \leq n - 1 + U_{n-1}(n-1).
\end{equation}
This alone can be used to show $U(n) \leq \binom{n}{2}$, and starting with any such $(t,t)$-shuffle gives the trivial upper bound
\begin{equation}
	\label{eqn:kn-triv}
	U_t(n) \leq \binom{t}{2} + t(n-t) = tn - \binom{t+1}{2}.
\end{equation}

As well as a sweeping construction which introduces extra positions, there is a simple divide and conquer construction which introduces extra counters. Substituting in the trivial bound (\ref{eqn:kn-triv}) for $U_{t-m}(n-m)$ and $U_m(n)$ in this construction, gives the trivial bound for $U_t(n)$. 

\begin{lemma}
	\label{lem:simple-divide}
	Let $0 \leq m \leq t$. Then 
	\[U_t(n) \leq U_{t-m}(n- m) + U_m(n)\]
\end{lemma}

\begin{proof}[Sketch proof]
The following gives a $(t,n)$-shuffle.
    \begin{enumerate}
        \item Shuffle the counters in positions $m+1, \dots, t$ across the positions $m  + 1, \dots, n$.
        \item Shuffle the counters in positions $1, \dots, m$ across all $n$ positions $1, \dots, n$.
    \end{enumerate}
\end{proof}

The last construction in this section is another divide and conquer approach which gives the bound \(U_{t}(n) \leq 2U(t) + t + U_t(n-t). \)
In particular, combining this with Lemma \ref{lem:simple-divide}, gives the upper bound
\begin{equation}
	\label{eqn:10+11}
	U(n) \leq 2 U(t) + 2 U(n-t) + t.
\end{equation}

\begin{lemma}[\cite{stack, angel2018perfect}]
	\label{lem:divide-and-conquer}
	Let $t \leq n/2$. Then 
	\[U_{t}(n) \leq 2U(t) + t + U_t(n-t). \]
\end{lemma}

\begin{proof}[Sketch proof]
	We give a construction with 3 stages.
	\begin{enumerate}
		\item Uniformly shuffle the counters in positions $1, \dots, t$. 
		\item Append the transpositions $(1, t +1, p_1), \dots, (t, 2t, p_{t})$ where the probabilities $p_{1}, \dots, p_t$ are chosen such that 
		\begin{equation}
			\label{eqn:divide-and-conquer}
			\prod_{i=1}^t (1 + p_ix) = \sum_{i=0}^t \frac{\binom{n-t}{i}\binom{t}{t -i}}{\binom{n}{t}} (1+x)^i.
		\end{equation}
		Such probabilities were shown to exist by Hui and Park \cite{hui2014representation}.
		\item Uniformly shuffle the positions $1, \dots, t$ and, separately, uniformly shuffle the positions $t+1, \dots, 2t$ into ${t+1}, \dots, n$.
	\end{enumerate}
\end{proof}

Using these constructions we now prove Lemma \ref{lem:constant}.

\begin{proof}[Proof of Lemma \ref{lem:constant}]
	Let $n = 2^a n_0 + b$ where $a,b \in \mathbb{N}$ and $0 \leq b < 2^a n_0$. Start with an $(n_0, n_0)$-shuffle using $U(n_0) \leq (1 - \eps) \binom{n_0}{2}$ transpositions and apply (\ref{eqn:10+11}) to see that 
	\[U(2n_0) \leq 4 U(n_0) + n_0 \leq \binom{2n_0}{2}  - 4\eps \binom{n_0}{2}.\]
	Applying (\ref{eqn:10+11}) a further $a-1$ times, we find   
	\[ U(2^a n_0) \leq \binom{2^an_0}{2}  - 4^a\eps \binom{n_0}{2}.\]
	
	Next, apply Lemma~\ref{lem:sweeping}, the sweeping construction, a total of $b$ times to get an $(n,n)$-transposition shuffle using at most
	\[ \binom{n}{2} - 4^a \eps \binom{n_0}{2} = \left( 1 - \frac{4^a \eps n_0 (n_0-1)}{n(n-1)} \right) \binom{n}{2}\]
	transpositions. For a fixed value of $a$, the worst constant is obtained when $b = 2^a n_0 - 1$ and substituting in $n = 2^{a+1}n_0 - 1$ gives
	\[U(n) \leq \left( 1 - \frac{4^a \eps n_0 (n_0-1)}{(2^{a+1}n_0 - 1)(2^{a+1}n_0 - 2)} \right) \binom{n}{2}.\]
	The result now follows since
	\[ \frac{4^a n_0 (n_0-1)}{(2^{a+1}n_0 - 1)(2^{a+1}n_0 - 2)} \geq \frac{n_0 - 1}{4 n_0}
	\]
	for all $n_0 \geq 1$.
\end{proof}

We remark that we have not tried to optimise the constant $\delta$ in this lemma and it may be possible to improve it by combining the constructions above in a better way. However, any bound from such a lemma will always be quite crude as it only takes into account a single improved construction for a single value of $n_0$.

\section{A new divide and conquer construction}
\label{sec:un}
We start by outlining and verifying a divide and conquer strategy which produces a $(t,mk)$-shuffle by splitting the $n$ positions into $m$ groups of $k$ positions each (where $k \geq t$). As in the previous constructions, it does not specify every transposition and instead focuses on the conditions which need to be satisfied at the end of each stage. The main results in this paper all follow from filling in this general construction or slight variations of it, and then combining the efficient shuffles with the constructions in Section \ref{sec:preliminaries}. 

Construction \ref{constr:main} is similar to the construction used in the proof of Lemma \ref{lem:divide-and-conquer}, particularly when $m = 2$ and $t = k$. In this case, the main difference is the introduction of a fourth stage which shuffles counters between the groups. This final stage allows us to relax condition (\ref{eqn:divide-and-conquer}) in the previous construction. If $W$ is the number of counters that move between the two groups, then condition (\ref{eqn:divide-and-conquer}) is equivalent to
\[\mathbb{P} (W = w) = \frac{\binom{k}{w}\binom{k}{k-w}}{\binom{2k}{k}}.\]
Using Construction \ref{constr:main}, we can replace (\ref{eqn:divide-and-conquer}) by the weaker condition 
\begin{equation*}
	\mathbb{P}(W = w) + \mathbb{P}(W = k- w) =\frac{2\binom{k}{w} \binom{k}{k-w}}{\binom{2k}{k}},
\end{equation*}
at the expense of some extra transpositions later on. When $k = 3, 5, 6$, the savings from using this weaker condition outweigh the extra transpositions, and we obtain shuffles which beat the trivial bounds. 

\begin{construction}
	\label{constr:main}
	Let $n = mk$ for two integers $m, k \geq 2$ and fix $t \leq k$. The construction consists of four stages.
	\begin{enumerate}
		\item Uniformly shuffle the $t$ counters in positions $1, \dots, t$.
		\item Perform any sequence of lazy transpositions which satisfies the following condition. For $i\in [m]$, let $W_i$ be the number of counters in positions ${(i-1)k + 1}, \dots, {ik}$ at the end of the chosen sequence, and let $W = (W_1, \dots, W_m)$; then  
		\begin{equation}
			\label{eqn-cond}
			\sum_{\sigma \in S_m} \mathbb{P}\left( W = (w_{\sigma(1)}, \dots, w_{\sigma(m)}) \right) = m!\frac{\prod_{i=1}^m \binom{k}{w_i}}{\binom{n}{t}}
		\end{equation}
		for every choice of $(w_1, \dots, w_m)$ with $w_1 + \dotsb + w_m = t$.
		\item For each $i = 1, \dots, m$, uniformly shuffle the counters in positions ${(i-1)k + 1}$, ${(i-1)k + 2}, \dots, {ik}$ across those same positions, i.e. over the positions ${(i-1)k + 1},$ ${(i-1)k + 2},\dots, {ik}$.
		\item For each $j = 1, \dots, k$, uniformly shuffle the counters in the positions ${j}, {k+j}$,$\dots, {(m-1)k + j}$ across those same positions, i.e. across the positions ${j}, {k+j}, \dots, {(m-1)k + j}$.
	\end{enumerate}
\end{construction}

 We now show that any sequence of transpositions following this construction gives a valid $(t, mk)$-shuffle, and we will fill in the necessary stages for certain choices of $t$, $m$ and $k$ in the later sections.

\begin{theorem}
	\label{thm:constr}
	If a sequence $T_1, \dots, T_\ell$ of transpositions satisfies the conditions in  Construction \ref{constr:main}, then the sequence forms a $(t,mk)$-shuffle.
\end{theorem}
\begin{proof}
	
	After stage 1, we may view the counters as indistinguishable. Indeed, stage 1 defines a uniformly random permutation $\sigma$ and 
	stages 2, 3 and 4 define a (random) injective map $F$ from $\{1,\dots, t\}$ to $\{1, \dots, n\}$.  The probability that the counters end in the ordered positions $(v_1, \dots, v_t)$ is given by 
	\[ \sum_{f: \Image(f) = \{v_1, \dots, v_t\}} \mathbb{P}(F = f) \mathbb{P}(f^{-1}(v_i) = \sigma(i) \text{ for all $i$}), \]
	and $\mathbb{P}(f^{-1}(v_i) = \sigma(i) \text{ for all $i$}) = 1/t!$ as $\sigma$ is uniformly random. Hence, the probability only depends on the set $\{v_1, \dots, v_t\}$, and we may view the counters as indistinguishable.
	
	We now show that every set of positions appears with equal probability. Denote the random number of  counters in positions ${j}, {k+j}, \dots, {(m-1)k + j}$ at the start of stage 4 by $\widetilde{W}_j$, and let $\widetilde{W} = \big( \widetilde{W}_1, \dots, \widetilde{W}_k\big)$. Since the last step shuffles the counters across ${j}, {k+j}, \dots, {(m-1)k + j}$ uniformly for every $j$, any two patterns which have the same number of counters in the positions ${j}, {k+j}, \dots, {(m-1)k + j}$ for every $j$, occur with equal probability. Hence, to ensure that every set of positions appears with equal probability we need
	\[ \mathbb{P}\left(\widetilde{W} = (\widetilde{w}_1, \dots, \widetilde{w}_k) \right) = \frac{\prod_{i=1}^k \binom{m}{ \widetilde{w}_i}}{\binom{n}{t}}\] for all $\widetilde{w} = (\widetilde{w}_1, \dots, \widetilde{w}_k)$ such that $\widetilde{w}_1 +  \dotsb + \widetilde{w}_k = t$.

	By construction, the distribution of $\widetilde{W}$ is entirely determined by the positions of the counters at the end of stage 2. In fact, by our choice of stage 3, it is determined by the vector $W = (W_1, \dots, W_n)$ where $W_i$ is the number of counters across the positions ${(i-1)k + 1}, \dots, {ik}$ at the end of stage 2. Further, due to the symmetry of the $m$ sets of positions, the distribution is determined by the unordered multiset $\{W_i : i = 1, \dots, m\}$. This means, we can guarantee that $\widetilde{W}$ has the correct distribution, and therefore that every set of $t$ positions appears with equal probability, provided the unordered multiset has the correct distribution. That is, we require
	\begin{equation}\label{eqn:pw}
		p_w := \mathbb{P}\left( W \in \{ (w_{\sigma(1)}, \dots, w_{\sigma(m)}) : \sigma \in S_m\} \right)
	\end{equation} to have the correct value for all $w = (w_1, \dots, w_m)$ such that $w_1 + \dotsb w_m = t$.
	
	We claim that this is implied by the condition (\ref{eqn-cond}). To do so we calculate the correct value for $p_w$ by considering the case where we have the stronger condition that \[\mathbb{P}\left(W = (w_1, \dots, w_m) \right) = \frac{\prod_{i=1}^m \binom{k}{w_i}}{\binom{n}{t}}\] for every $w$ such that $w_1 + \dots + w_m = t$. After shuffling the counters in stage 3, the counters are uniformly distributed over all $n$ positions. Clearly, applying any lazy transposition to the uniform distribution also gives the uniform distribution, and the counters are also uniformly distributed at the end of stage 4. This implies that we should take \[p_w  = \left|\{ (w_{\sigma(1)}, \dots, w_{\sigma(m)}) : \sigma \in S_n\} \right|  \frac{\prod_{i=1}^m \binom{k}{w_i}}{\binom{n}{t}},\] and this is  equivalent to the condition (\ref{eqn-cond}).
\end{proof}

\subsection{Splitting into two groups}
\label{sec:2parts}

Construction \ref{constr:main} only gives conditions on the output of each stage and it does not specify the exact transpositions. In this section we consider the case where $m = 2$, and give improved bounds on $U_3(n)$ and $U_4(n)$, from which efficient $U_t(n)$ shuffles can be found. Construction \ref{constr:main} can also be applied when shuffling $5$ and $6$ counters for small values of $k$, but the bounds obtained are worse than combining the efficient $U_3(n)$ shuffles with the constructions in Section \ref{sec:preliminaries}.

Throughout this section, we will work with sequences of transpositions which shuffle indistinguishable counters, and we denote the minimum number of transpositions needed to shuffle $t$ indistinguishable counters over $n$ positions by $\widehat{U}_t(n)$. This is useful when considering Construction \ref{constr:main} in the case where $k > t$. At the end of stage 2, the positions $1, \dots, k$ fall into 3 categories: some positions must contain a counter, some positions cannot contain a counter and some positions may or may not contain a counter. Since the counters are made indistinguishable by stage 1, the positions which definitely must contain are equivalent and we do not need to shuffle the counters in these positions together. By working with $\widehat{U}_t(n)$ instead of $U_t(n)$, we can exploit this to save transpositions.

On the other hand, it is easy to create a $(t,n)$-shuffle from a shuffle which only shuffles $t$ indistinguishable counters over $n$ positions by first shuffling the $t$ counters amongst themselves. Hence, we get the bound 
\[U_t(n) \leq U(t) + \widehat{U}_t(n), \]
and efficient shuffles for indistinguishable counters immediately lead to efficient shuffles for distinguishable counters.

\begin{figure}
	\centering
	\begin{tikzpicture}[xscale=0.5\textwidth/12cm, yscale=3/4]
	\draw (0, 0) -- (12, 0);
	\draw (0, 0.8) -- (12, 0.8);
	\draw (0, 1.6) -- (12, 1.6);
	\draw (0, 2.4000000000000004) -- (12, 2.4000000000000004);
	\draw (0, 3.2) -- (12, 3.2);
	\draw (0, 4) -- (12, 4);

	\draw (0.9230769230769231, 3.2) -- (0.9230769230769231, 2.4000000000000004);
	\draw node[style=vertex] at (0.9230769230769231, 3.2) {};
	\draw node[style=vertex] at (0.9230769230769231, 2.4000000000000004) {};
	\draw (1.8461538461538463, 4) -- (1.8461538461538463, 2.4000000000000004);
	\draw node[style=vertex] at (1.8461538461538463, 4) {};
	\draw node[style=vertex] at (1.8461538461538463, 2.4000000000000004) {};
	\draw (2.769230769230769, 3.2) -- (2.769230769230769, 2.4000000000000004);
	\draw node[style=vertex] at (2.769230769230769, 3.2) {};
	\draw node[style=vertex] at (2.769230769230769, 2.4000000000000004) {};
	\draw (4.615384615384616, 2.4000000000000004) -- (4.615384615384616, 1.6);
	\draw node[style=vertex] at (4.615384615384616, 2.4000000000000004) {};
	\draw node[style=vertex] at (4.615384615384616, 1.6) {};
	\draw (6.461538461538462, 3.2) -- (6.461538461538462, 2.4000000000000004);
	\draw node[style=vertex] at (6.461538461538462, 3.2) {};
	\draw node[style=vertex] at (6.461538461538462, 2.4000000000000004) {};
	\draw (6.461538461538462, 1.6) -- (6.461538461538462, 0.8);
	\draw node[style=vertex] at (6.461538461538462, 1.6) {};
	\draw node[style=vertex] at (6.461538461538462, 0.8) {};
	\draw (7.384615384615385, 4) -- (7.384615384615385, 2.4000000000000004);
	\draw node[style=vertex] at (7.384615384615385, 4) {};
	\draw node[style=vertex] at (7.384615384615385, 2.4000000000000004) {};
	\draw (7.384615384615385, 1.6) -- (7.384615384615385, 0);
	\draw node[style=vertex] at (7.384615384615385, 1.6) {};
	\draw node[style=vertex] at (7.384615384615385, 0) {};
	\draw (9.230769230769232, 2.4000000000000004) -- (9.230769230769232, 1.6);
	\draw node[style=vertex] at (9.230769230769232, 2.4000000000000004) {};
	\draw node[style=vertex] at (9.230769230769232, 1.6) {};
	\draw (10.153846153846155, 3.2) -- (10.153846153846155, 0.8);
	\draw node[style=vertex] at (10.153846153846155, 3.2) {};
	\draw node[style=vertex] at (10.153846153846155, 0.8) {};
	\draw (11.076923076923077, 4) -- (11.076923076923077, 0);
	\draw node[style=vertex] at (11.076923076923077, 4) {};
	\draw node[style=vertex] at (11.076923076923077, 0) {};
\end{tikzpicture}
	\caption{A $(3,6)$-shuffle with just 11 transpositions constructed as in the proof of Lemma \ref{lem:indist}.}
	\label{fig:3-6}
\end{figure}
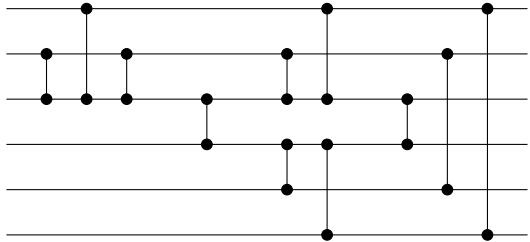

\begin{figure}
	\begin{tikzpicture}[xscale=\textwidth/12cm]
	\draw (0, 0) -- (12, 0);
	\draw (0, 0.36363636363636365) -- (12, 0.36363636363636365);
	\draw (0, 0.7272727272727273) -- (12, 0.7272727272727273);
	\draw (0, 1.0909090909090908) -- (12, 1.0909090909090908);
	\draw (0, 1.4545454545454546) -- (12, 1.4545454545454546);
	\draw (0, 1.8181818181818183) -- (12, 1.8181818181818183);
	\draw (0, 2.1818181818181817) -- (12, 2.1818181818181817);
	\draw (0, 2.5454545454545454) -- (12, 2.5454545454545454);
	\draw (0, 2.909090909090909) -- (12, 2.909090909090909);
	\draw (0, 3.272727272727273) -- (12, 3.272727272727273);
	\draw (0, 3.6363636363636367) -- (12, 3.6363636363636367);
	\draw (0, 4) -- (12, 4);

	\draw (0.5217391304347826, 4) -- (0.5217391304347826, 3.6363636363636367);
	\draw node[style=vertex] at (0.5217391304347826, 4) {};
	\draw node[style=vertex] at (0.5217391304347826, 3.6363636363636367) {};
	\draw (1.0434782608695652, 4) -- (1.0434782608695652, 3.272727272727273);
	\draw node[style=vertex] at (1.0434782608695652, 4) {};
	\draw node[style=vertex] at (1.0434782608695652, 3.272727272727273) {};
	\draw (1.5652173913043477, 4) -- (1.5652173913043477, 3.6363636363636367);
	\draw node[style=vertex] at (1.5652173913043477, 4) {};
	\draw node[style=vertex] at (1.5652173913043477, 3.6363636363636367) {};
	\draw (2.608695652173913, 4) -- (2.608695652173913, 1.8181818181818183);
	\draw node[style=vertex] at (2.608695652173913, 4) {};
	\draw node[style=vertex] at (2.608695652173913, 1.8181818181818183) {};
	\draw (3.652173913043478, 4) -- (3.652173913043478, 3.6363636363636367);
	\draw node[style=vertex] at (3.652173913043478, 4) {};
	\draw node[style=vertex] at (3.652173913043478, 3.6363636363636367) {};
	\draw (4.173913043478261, 4) -- (4.173913043478261, 3.272727272727273);
	\draw node[style=vertex] at (4.173913043478261, 4) {};
	\draw node[style=vertex] at (4.173913043478261, 3.272727272727273) {};
	\draw (5.217391304347826, 3.272727272727273) -- (5.217391304347826, 2.909090909090909);
	\draw node[style=vertex] at (5.217391304347826, 3.272727272727273) {};
	\draw node[style=vertex] at (5.217391304347826, 2.909090909090909) {};
	\draw (5.217391304347826, 1.8181818181818183) -- (5.217391304347826, 1.4545454545454546);
	\draw node[style=vertex] at (5.217391304347826, 1.8181818181818183) {};
	\draw node[style=vertex] at (5.217391304347826, 1.4545454545454546) {};
	\draw (5.739130434782608, 3.6363636363636367) -- (5.739130434782608, 3.272727272727273);
	\draw node[style=vertex] at (5.739130434782608, 3.6363636363636367) {};
	\draw node[style=vertex] at (5.739130434782608, 3.272727272727273) {};
	\draw (5.739130434782608, 2.909090909090909) -- (5.739130434782608, 2.5454545454545454);
	\draw node[style=vertex] at (5.739130434782608, 2.909090909090909) {};
	\draw node[style=vertex] at (5.739130434782608, 2.5454545454545454) {};
	\draw (5.739130434782608, 1.8181818181818183) -- (5.739130434782608, 1.0909090909090908);
	\draw node[style=vertex] at (5.739130434782608, 1.8181818181818183) {};
	\draw node[style=vertex] at (5.739130434782608, 1.0909090909090908) {};
	\draw (6.260869565217391, 4) -- (6.260869565217391, 3.272727272727273);
	\draw node[style=vertex] at (6.260869565217391, 4) {};
	\draw node[style=vertex] at (6.260869565217391, 3.272727272727273) {};
	\draw (6.260869565217391, 2.909090909090909) -- (6.260869565217391, 2.1818181818181817);
	\draw node[style=vertex] at (6.260869565217391, 2.909090909090909) {};
	\draw node[style=vertex] at (6.260869565217391, 2.1818181818181817) {};
	\draw (6.260869565217391, 1.8181818181818183) -- (6.260869565217391, 0.7272727272727273);
	\draw node[style=vertex] at (6.260869565217391, 1.8181818181818183) {};
	\draw node[style=vertex] at (6.260869565217391, 0.7272727272727273) {};
	\draw (6.782608695652174, 3.272727272727273) -- (6.782608695652174, 2.909090909090909);
	\draw node[style=vertex] at (6.782608695652174, 3.272727272727273) {};
	\draw node[style=vertex] at (6.782608695652174, 2.909090909090909) {};
	\draw (6.782608695652174, 1.8181818181818183) -- (6.782608695652174, 0.36363636363636365);
	\draw node[style=vertex] at (6.782608695652174, 1.8181818181818183) {};
	\draw node[style=vertex] at (6.782608695652174, 0.36363636363636365) {};
	\draw (7.304347826086956, 3.6363636363636367) -- (7.304347826086956, 2.5454545454545454);
	\draw node[style=vertex] at (7.304347826086956, 3.6363636363636367) {};
	\draw node[style=vertex] at (7.304347826086956, 2.5454545454545454) {};
	\draw (7.304347826086956, 1.8181818181818183) -- (7.304347826086956, 0);
	\draw node[style=vertex] at (7.304347826086956, 1.8181818181818183) {};
	\draw node[style=vertex] at (7.304347826086956, 0) {};
	\draw (7.826086956521739, 4) -- (7.826086956521739, 2.1818181818181817);
	\draw node[style=vertex] at (7.826086956521739, 4) {};
	\draw node[style=vertex] at (7.826086956521739, 2.1818181818181817) {};
	\draw (8.869565217391305, 2.1818181818181817) -- (8.869565217391305, 1.8181818181818183);
	\draw node[style=vertex] at (8.869565217391305, 2.1818181818181817) {};
	\draw node[style=vertex] at (8.869565217391305, 1.8181818181818183) {};
	\draw (9.391304347826086, 2.5454545454545454) -- (9.391304347826086, 1.4545454545454546);
	\draw node[style=vertex] at (9.391304347826086, 2.5454545454545454) {};
	\draw node[style=vertex] at (9.391304347826086, 1.4545454545454546) {};
	\draw (9.91304347826087, 2.909090909090909) -- (9.91304347826087, 1.0909090909090908);
	\draw node[style=vertex] at (9.91304347826087, 2.909090909090909) {};
	\draw node[style=vertex] at (9.91304347826087, 1.0909090909090908) {};
	\draw (10.434782608695652, 3.272727272727273) -- (10.434782608695652, 0.7272727272727273);
	\draw node[style=vertex] at (10.434782608695652, 3.272727272727273) {};
	\draw node[style=vertex] at (10.434782608695652, 0.7272727272727273) {};
	\draw (10.956521739130434, 3.6363636363636367) -- (10.956521739130434, 0.36363636363636365);
	\draw node[style=vertex] at (10.956521739130434, 3.6363636363636367) {};
	\draw node[style=vertex] at (10.956521739130434, 0.36363636363636365) {};
	\draw (11.478260869565217, 4) -- (11.478260869565217, 0);
	\draw node[style=vertex] at (11.478260869565217, 4) {};
	\draw node[style=vertex] at (11.478260869565217, 0) {};
\end{tikzpicture}
	\caption{A $(3,12)$-shuffle using just 25 transpositions constructed as in Lemma \ref{lem:indist}.}
	\label{fig:3-12}
\end{figure}
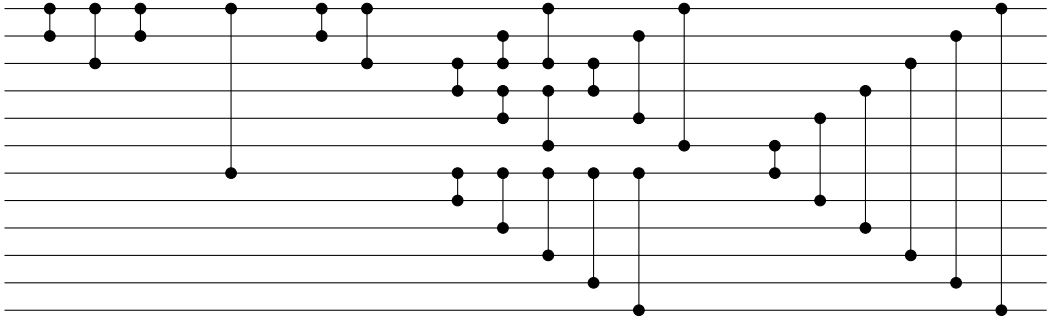

\begin{lemma}
	\label{lem:indist}
		If $k \geq 3$, then 
		\begin{align*}
			\widehat{U}_3(2k) &\leq 2k + 2 + \widehat{U}_3(k),\\
			\intertext{and if $k\geq 4$,}
			\widehat{U}_4(2k) &\leq 3k + 4 + \widehat{U}_4(k).\\
		\end{align*}
	\end{lemma}

	\begin{proof}
		We use Construction \ref{constr:main} in the case where $m = 2$ except, since we are starting with indistinguishable counters, stage 1 is unnecessary and we start with stage 2. We only give a proof for the case where $t = 4$ since the case $t = 3$ is similar, although slightly easier.
		
		Start stage 2 with the transpositions $(1, k + 1, p_1)$ and $(2, k +2, p_2)$. 
		The condition (\ref{eqn-cond}) becomes
		\begin{equation*}
		    \mathbb{P}(W_2 = w) + \mathbb{P}(W_2 = t - w) =\frac{2\binom{k}{w} \binom{k}{t -w}}{\binom{2k}{t}}
	    \end{equation*}
	for $w = 0, 1, 2$, and this can be further simplified to
		\begin{equation}
		    \label{eqn:cond-2}
			\mathbb{P}(W_2 = w) = \begin{cases}
				\dfrac{2\binom{k}{w} \binom{k}{t-w}}{\binom{2k}{t}} & w \neq 2,\vspace{0.3cm} \\
				\dfrac{\binom{k}{w}^2}{\binom{2k}{t}}  & w = 2.
			\end{cases}
		\end{equation}
	To compute the values of $p_1$ and $p_2$, we use the approach of Hui and Park in \cite{hui2014representation} and consider the generating function $\mathbb{E}\left[(1+x)^{W_2}\right]$. As the sum of independent Bernoulli random variables we have
	\[\mathbb{E} \left[ (1+x)^{W_2}\right] = (1 + p_1x)(1 + p_2x). \]
	We may also write
	\[\mathbb{E}\left[(1+x)^{W_2} \right] = \sum_{w=0}^{2} \mathbb{P}\left(W_2 = w \right) (1+x)^w, \]
	where the probabilities are as given in (\ref{eqn:cond-2}), and we can compare the roots of these polynomials to find the values of $p_i$. We have
		\[\expec{(1+x)^{W_2}} = 1 + \frac{k(5k -7)}{(2k-1)(2k-3)} x + \frac{3k(k-1)}{2(2k-1)(2k-3)} x^2,\]
		which has determinant 
		\[\frac{k(k-2)(k^2 + 4k -9)}{(2k-3)^2(2k-1)^2}.\]
		This is positive for all $k > 2$, and hence, the polynomial has 2 roots. If the two roots are $r_1$ and $r_2$, take $p_1 = - 1/r_1$ and $p_2 = - 1/r_2$. Since the polynomial is clearly positive for all $x \geq -1$, both $p_1$ and $p_2$ are valid probabilities.
		
		In stage 3, first shuffle the positions $1$ and $2$ into the positions $1, 2, 3, 4$ using $U_2(4) = 5$ transpositions. Given there are $w$ counters left in the positions $1, 2, 3, 4$, they are in a uniformly random order and in uniformly random positions. Hence, we can view the positions as containing 4 indistinguishable counters and shuffle the counters across the positions $1, 2, \dots, k$ using $\widehat{U}_4(k)$ transpositions. Shuffle any counters in positions $k +1$ and $k+2$ across the positions $k +1, \dots, 2k$ using $U_2(k)$ transpositions.
		
		Finally, we end with the $k$ transpositions $(k, k+1, 1/2), \dots, (1, 2k, 1/2)$ in stage 4. This shuffles 4 indistinguishable counters across $2k$ positions using
		\[7 + \widehat{U}_4(k) + U_2(k) + k\]
		transpositions, and the bound follows using $U_2(k) \leq 2k -3$.
	\end{proof}

Since $\widehat{U}_3(3) = 0$, this construction gives a (3,6)-shuffle which uses 11 transpositions and beats the the trivial bound, and this can be used to prove the first part of Theorem \ref{thm:wrongconstant}. Applying Lemma \ref{lem:simple-divide} gives the upper bound $U(6) \leq 14$, and Lemma \ref{lem:constant} shows that there is some $\eps > 0$ such that $U(n) \leq (1- \eps)\binom{n}{2}$ for all $n \geq 6$.

If we do not worry about the small savings obtained by noting that some positions with counters are already indistinguishable, we can replace stage 3 with a $(t,k)$-shuffle and a $(\floor{t/2}, k)$-shuffle. This gives the simple bounds,

\begin{align*}
    U_3(2k) &\leq 2k + 3 + U_3(k),\\
    U_4(2k) &\leq 3k + 5 + U_4(k).
\end{align*}

Substituting in the trivial bounds for $U_3(k)$ and $U_4(k)$ gives the following.
\begin{align*}
    U_3(2k) &\leq 5k -3,\\
    U_4(2k) &\leq 7k -5.
\end{align*}
Although these bounds can be improved by repeated applications of the construction, they already guarantee a constant-factor improvement whenever $n$ is even, which is a remarkably weak condition. 

\begin{corollary}
	\label{cor:asymptotics}
		We have 
		\begin{align*}
			\widehat{U}_3(n) &\leq 2n + 5 \log_2 (n),\\
			\widehat{U}_4(n) &\leq 3n + 8 \log_2 (n).
		\end{align*}
		Therefore,
		\begin{align*}
			U_3(n) &\leq 2n + 5 \log_2 (n) + 3,\\
			U_4(n) &\leq 3n + 8 \log_2 (n)  + 6.
		\end{align*}
	\end{corollary}
Since we are primarily interested in the asymptotics of $U_t(n)$ as $n$ tends to infinity, we have not attempted to optimise the lower order terms. In fact, the bounds are quite weak for small $n$ and start off considerably worse than the trivial bounds.

\begin{proof}
We first prove $\widehat{U}_3(n) \leq 2n+5\log_2(n)$. The bound is true for $n \leq 34$ as $2n + 5 \log_2 (n) \geq 3n - 9$, the trivial upper bound for $\widehat{U}_3(n)$. If $n$ is odd, start with the transpositions $(1, n, 3/n)$, followed by the transpositions $(1, 2, 1/3)$ and $(1,3, 1/2)$. The latter transpositions shuffle the position $1$ into positions $1, 2, 3$, making the positions again indistinguishable. It is now sufficient to distribute the three indistinguishable positions $1, 2, 3$ over $1, 2, \dots, n-1$. Let $k = \floor{n/2}$ so that $n = 2k + x$ for some $0 \leq x < 2$. Then, by induction,
    \begin{align*}
        \widehat{U}_3(n) &\leq 3 + \widehat{U}_3(2k)\\
        &\leq 3 + 2k + 2 + \widehat{U}_3(k)\\
        &\leq 4k + 5 \log_2 (k) + 5\\
        &= 4k + 5\log_2(2k)\\
        &\leq 2n + 5 \log_2 (n).
    \end{align*}
    The inequality for $\widehat{U}_4(n)$ is shown in a similar fashion. The bound is true for $n \leq 64$ as $3n + 8 \log_2(n) \geq 4n - 16$ and, following the steps for the bound on $\widehat{U}_3(n)$, we obtain the recurrence
    \[\widehat{U}_4(n) \leq 4 + \widehat{U}_{4}(2k) \leq 8 + 3k + \widehat{U}_4(k),\]
    and the result follows.
    
    The second part of the theorem follows immediately from the bound $U_t(n) \leq U(t) + \widehat{U}_t(n)$ and the trivial bound $U(t) \leq \binom{t}{2}$.
\end{proof}

We are now armed with the results needed to prove the second part of Theorem \ref{thm:wrongconstant}, which we do by repeatedly shuffling in 3 counters at a time.

\begin{proof}[Proof of Theorem \ref{thm:wrongconstant}]

The first part of the theorem follows from the case $k = 3$ in Lemma \ref{lem:indist} and results from Section \ref{sec:preliminaries}.

To prove the second part, we first bound $U_{3\ell}(n)$. Applying Lemma \ref{lem:simple-divide} and Corollary \ref{cor:asymptotics} we find
\begin{align*}
    U_{3\ell}(n) &\leq U_3(n - 3(\ell -1)) + U_{3(\ell - 1)}(n)\\
    &\leq 2(n - 3 (\ell - 1)) + 5 \log_2 (n) + 3 + U_{3(\ell -1)}(n).
\end{align*}
Solving this recursion relation we find
\[U_{3\ell}(n) \leq 2\ell n + 5 \ell \log_2 (n)  + 6 \ell - 3\ell^2.\]

Let $\ell = \floor{n/3}$ and write $n = 3 \ell + r$ where $0 \leq r < 3 $. Then
\begin{align*}
	U(n) &\leq U_r(r) + U_{3\ell} (n)\\
    &\leq 1 + 2\ell n + 5 \ell \log_2 (n)  + 6 \ell - 3\ell^2\\
    &\leq \frac{n^2}{3} + \frac{5 n \log_2 (n)}{3} + 2n + 1\\
    &= \frac{2}{3} \binom{n}{2} + O(n \log (n)).\qedhere
\end{align*}
\end{proof}

The second part of Theorem \ref{thm:kn} will be shown in a similar way. We could prove the first part using the results in this section, but instead we use the efficient $(3, 3k)$-shuffle given in the next section. Although the shuffle is less efficient than the shuffles in this section, it makes the proof marginally easier as the parity modulo 3 is preserved throughout the proof.

\subsection{Splitting into three or more groups}
\label{sec:kn}

In this section we consider Construction \ref{constr:main} in the case where $m \geq 3$ and give efficient $(3, n)$- and $(4, n)$-shuffles. Although the construction in Lemma \ref{lem:indist} leads to more efficient shuffles in general, both of these constructions will be useful: the efficient $(3,3k)$-shuffles given in Lemma \ref{lem:3} will be used in the proof of Theorem \ref{thm:kn}, and the efficient $(4,3k)$-shuffles given in Lemma \ref{lem:4} will be used in Section \ref{sec:improved-star} to give efficient shuffles consisting entirely of star transpositions. Examples of the shuffles are given in Figure \ref{fig:3-15} and Figure \ref{fig:4-12} respectively.

\begin{figure}
	\centering
	\begin{tikzpicture}[xscale=\textwidth/12cm]
	\draw (0, 0) -- (12, 0);
	\draw (0, 0.2857142857142857) -- (12, 0.2857142857142857);
	\draw (0, 0.5714285714285714) -- (12, 0.5714285714285714);
	\draw (0, 0.8571428571428571) -- (12, 0.8571428571428571);
	\draw (0, 1.1428571428571428) -- (12, 1.1428571428571428);
	\draw (0, 1.4285714285714284) -- (12, 1.4285714285714284);
	\draw (0, 1.7142857142857142) -- (12, 1.7142857142857142);
	\draw (0, 2) -- (12, 2);
	\draw (0, 2.2857142857142856) -- (12, 2.2857142857142856);
	\draw (0, 2.571428571428571) -- (12, 2.571428571428571);
	\draw (0, 2.8571428571428568) -- (12, 2.8571428571428568);
	\draw (0, 3.142857142857143) -- (12, 3.142857142857143);
	\draw (0, 3.4285714285714284) -- (12, 3.4285714285714284);
	\draw (0, 3.714285714285714) -- (12, 3.714285714285714);
	\draw (0, 4) -- (12, 4);

	\draw (0.36363636363636365, 4) -- (0.36363636363636365, 3.714285714285714);
	\draw node[style=vertex] at (0.36363636363636365, 4) {};
	\draw node[style=vertex] at (0.36363636363636365, 3.714285714285714) {};
	\draw (0.7272727272727273, 4) -- (0.7272727272727273, 3.4285714285714284);
	\draw node[style=vertex] at (0.7272727272727273, 4) {};
	\draw node[style=vertex] at (0.7272727272727273, 3.4285714285714284) {};
	\draw (1.0909090909090908, 4) -- (1.0909090909090908, 3.714285714285714);
	\draw node[style=vertex] at (1.0909090909090908, 4) {};
	\draw node[style=vertex] at (1.0909090909090908, 3.714285714285714) {};
	\draw (1.8181818181818183, 4) -- (1.8181818181818183, 2.571428571428571);
	\draw node[style=vertex] at (1.8181818181818183, 4) {};
	\draw node[style=vertex] at (1.8181818181818183, 2.571428571428571) {};
	\draw (2.1818181818181817, 3.714285714285714) -- (2.1818181818181817, 1.1428571428571428);
	\draw node[style=vertex] at (2.1818181818181817, 3.714285714285714) {};
	\draw node[style=vertex] at (2.1818181818181817, 1.1428571428571428) {};
	\draw (2.909090909090909, 4) -- (2.909090909090909, 3.714285714285714);
	\draw node[style=vertex] at (2.909090909090909, 4) {};
	\draw node[style=vertex] at (2.909090909090909, 3.714285714285714) {};
	\draw (2.909090909090909, 2.571428571428571) -- (2.909090909090909, 2.2857142857142856);
	\draw node[style=vertex] at (2.909090909090909, 2.571428571428571) {};
	\draw node[style=vertex] at (2.909090909090909, 2.2857142857142856) {};
	\draw (2.909090909090909, 1.1428571428571428) -- (2.909090909090909, 0.8571428571428571);
	\draw node[style=vertex] at (2.909090909090909, 1.1428571428571428) {};
	\draw node[style=vertex] at (2.909090909090909, 0.8571428571428571) {};
	\draw (3.272727272727273, 4) -- (3.272727272727273, 3.4285714285714284);
	\draw node[style=vertex] at (3.272727272727273, 4) {};
	\draw node[style=vertex] at (3.272727272727273, 3.4285714285714284) {};
	\draw (3.272727272727273, 2.571428571428571) -- (3.272727272727273, 2);
	\draw node[style=vertex] at (3.272727272727273, 2.571428571428571) {};
	\draw node[style=vertex] at (3.272727272727273, 2) {};
	\draw (3.272727272727273, 1.1428571428571428) -- (3.272727272727273, 0.5714285714285714);
	\draw node[style=vertex] at (3.272727272727273, 1.1428571428571428) {};
	\draw node[style=vertex] at (3.272727272727273, 0.5714285714285714) {};
	\draw (3.6363636363636367, 4) -- (3.6363636363636367, 3.142857142857143);
	\draw node[style=vertex] at (3.6363636363636367, 4) {};
	\draw node[style=vertex] at (3.6363636363636367, 3.142857142857143) {};
	\draw (3.6363636363636367, 2.571428571428571) -- (3.6363636363636367, 1.7142857142857142);
	\draw node[style=vertex] at (3.6363636363636367, 2.571428571428571) {};
	\draw node[style=vertex] at (3.6363636363636367, 1.7142857142857142) {};
	\draw (3.6363636363636367, 1.1428571428571428) -- (3.6363636363636367, 0.2857142857142857);
	\draw node[style=vertex] at (3.6363636363636367, 1.1428571428571428) {};
	\draw node[style=vertex] at (3.6363636363636367, 0.2857142857142857) {};
	\draw (4, 4) -- (4, 3.714285714285714);
	\draw node[style=vertex] at (4, 4) {};
	\draw node[style=vertex] at (4, 3.714285714285714) {};
	\draw (4, 2.571428571428571) -- (4, 1.4285714285714284);
	\draw node[style=vertex] at (4, 2.571428571428571) {};
	\draw node[style=vertex] at (4, 1.4285714285714284) {};
	\draw (4, 1.1428571428571428) -- (4, 0);
	\draw node[style=vertex] at (4, 1.1428571428571428) {};
	\draw node[style=vertex] at (4, 0) {};
	\draw (4.363636363636363, 4) -- (4.363636363636363, 3.4285714285714284);
	\draw node[style=vertex] at (4.363636363636363, 4) {};
	\draw node[style=vertex] at (4.363636363636363, 3.4285714285714284) {};
	\draw (4.7272727272727275, 4) -- (4.7272727272727275, 2.8571428571428568);
	\draw node[style=vertex] at (4.7272727272727275, 4) {};
	\draw node[style=vertex] at (4.7272727272727275, 2.8571428571428568) {};
	\draw (5.090909090909091, 4) -- (5.090909090909091, 3.714285714285714);
	\draw node[style=vertex] at (5.090909090909091, 4) {};
	\draw node[style=vertex] at (5.090909090909091, 3.714285714285714) {};
	\draw (5.454545454545455, 4) -- (5.454545454545455, 3.4285714285714284);
	\draw node[style=vertex] at (5.454545454545455, 4) {};
	\draw node[style=vertex] at (5.454545454545455, 3.4285714285714284) {};
	\draw (5.818181818181818, 4) -- (5.818181818181818, 3.714285714285714);
	\draw node[style=vertex] at (5.818181818181818, 4) {};
	\draw node[style=vertex] at (5.818181818181818, 3.714285714285714) {};
	\draw (6.545454545454546, 4) -- (6.545454545454546, 2.571428571428571);
	\draw node[style=vertex] at (6.545454545454546, 4) {};
	\draw node[style=vertex] at (6.545454545454546, 2.571428571428571) {};
	\draw (6.909090909090909, 4) -- (6.909090909090909, 1.1428571428571428);
	\draw node[style=vertex] at (6.909090909090909, 4) {};
	\draw node[style=vertex] at (6.909090909090909, 1.1428571428571428) {};
	\draw (7.272727272727273, 4) -- (7.272727272727273, 2.571428571428571);
	\draw node[style=vertex] at (7.272727272727273, 4) {};
	\draw node[style=vertex] at (7.272727272727273, 2.571428571428571) {};
	\draw (7.636363636363637, 3.714285714285714) -- (7.636363636363637, 2.2857142857142856);
	\draw node[style=vertex] at (7.636363636363637, 3.714285714285714) {};
	\draw node[style=vertex] at (7.636363636363637, 2.2857142857142856) {};
	\draw (8, 3.714285714285714) -- (8, 0.8571428571428571);
	\draw node[style=vertex] at (8, 3.714285714285714) {};
	\draw node[style=vertex] at (8, 0.8571428571428571) {};
	\draw (8.363636363636363, 3.714285714285714) -- (8.363636363636363, 2.2857142857142856);
	\draw node[style=vertex] at (8.363636363636363, 3.714285714285714) {};
	\draw node[style=vertex] at (8.363636363636363, 2.2857142857142856) {};
	\draw (8.727272727272727, 3.4285714285714284) -- (8.727272727272727, 2);
	\draw node[style=vertex] at (8.727272727272727, 3.4285714285714284) {};
	\draw node[style=vertex] at (8.727272727272727, 2) {};
	\draw (9.090909090909092, 3.4285714285714284) -- (9.090909090909092, 0.5714285714285714);
	\draw node[style=vertex] at (9.090909090909092, 3.4285714285714284) {};
	\draw node[style=vertex] at (9.090909090909092, 0.5714285714285714) {};
	\draw (9.454545454545455, 3.4285714285714284) -- (9.454545454545455, 2);
	\draw node[style=vertex] at (9.454545454545455, 3.4285714285714284) {};
	\draw node[style=vertex] at (9.454545454545455, 2) {};
	\draw (9.818181818181818, 3.142857142857143) -- (9.818181818181818, 1.7142857142857142);
	\draw node[style=vertex] at (9.818181818181818, 3.142857142857143) {};
	\draw node[style=vertex] at (9.818181818181818, 1.7142857142857142) {};
	\draw (10.181818181818182, 3.142857142857143) -- (10.181818181818182, 0.2857142857142857);
	\draw node[style=vertex] at (10.181818181818182, 3.142857142857143) {};
	\draw node[style=vertex] at (10.181818181818182, 0.2857142857142857) {};
	\draw (10.545454545454545, 3.142857142857143) -- (10.545454545454545, 1.7142857142857142);
	\draw node[style=vertex] at (10.545454545454545, 3.142857142857143) {};
	\draw node[style=vertex] at (10.545454545454545, 1.7142857142857142) {};
	\draw (10.90909090909091, 2.8571428571428568) -- (10.90909090909091, 1.4285714285714284);
	\draw node[style=vertex] at (10.90909090909091, 2.8571428571428568) {};
	\draw node[style=vertex] at (10.90909090909091, 1.4285714285714284) {};
	\draw (11.272727272727273, 2.8571428571428568) -- (11.272727272727273, 0);
	\draw node[style=vertex] at (11.272727272727273, 2.8571428571428568) {};
	\draw node[style=vertex] at (11.272727272727273, 0) {};
	\draw (11.636363636363637, 2.8571428571428568) -- (11.636363636363637, 1.4285714285714284);
	\draw node[style=vertex] at (11.636363636363637, 2.8571428571428568) {};
	\draw node[style=vertex] at (11.636363636363637, 1.4285714285714284) {};
\end{tikzpicture}
	\caption[A $(3, 15)$-shuffle using just 37 transpositions constructed as in the proof of Lemma \ref{lem:3}.]{A $(3, 15)$-shuffle constructed as in the proof of Lemma \ref{lem:3} in the case $m = 3$ and $k = 5$. This uses $37$ transpositions, beating the trivial upper bound of $39$.}
	\label{fig:3-15}
\end{figure}
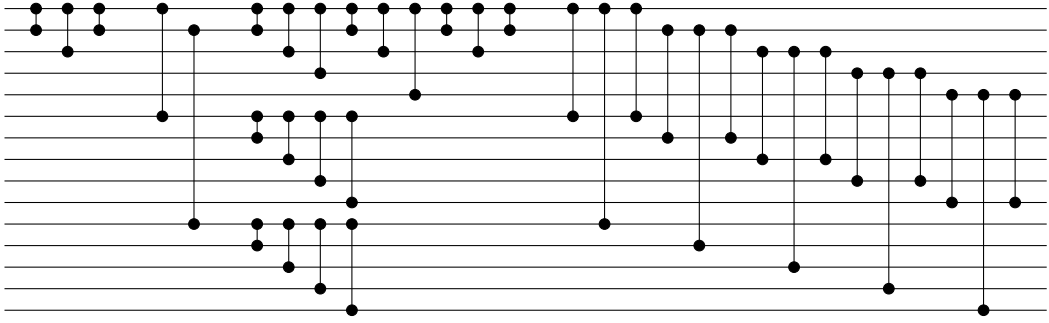

\begin{lemma}
	\label{lem:3}
	Let $m, k \geq 3$. Then
	\[U_3(mk) \leq 2k + 3 + U_3(k) + k U_3(m),\]
	and in particular,
	\[U_3(3k) \leq 5k + 3 + U_3(k).\]
\end{lemma}
\begin{proof}[Proof of Lemma \ref{lem:3}]
	We use Construction \ref{constr:main} in the case where $t = 3$.
	Start by shuffling the 3 counters using $3$ transpositions. In stage 2 we use the transpositions $(1,k + 1, x)$ and $(2,2k+ 1, y)$. There are only 3 possible options for the unordered multisets of the $W_i$, namely $\{3, 0, \dots, 0\}$, $\{2, 1, 0, \dots, 0\}$ and $\{1,1,1,0, \dots, 0\}$. Consider the first case, where $w_1 = 3$ and $w_2 = \dotsb = w_m = 0$. Then 
	\begin{align*}
	    \sum_{\sigma \in S_m} \mathbb{P}\left( W = (w_{\sigma(1)}, \dots, w_{\sigma(m)}) \right) &= \sum_{\substack{\sigma \in S_m\\ \sigma(1) = 1}} \mathbb{P}\left(W_{1} = 3 \right)\\
	    &= (m-1)! \cdot \mathbb{P}(W_1 = 3)\\
	    &= (m-1)! \cdot (1-x)(1-y).
	\end{align*}
	The other cases are similar, and the condition (\ref{eqn-cond}) becomes the following.
	\begin{align}
		(1-x)(1-y) &= \frac{m \binom{k}{3}}{\binom{mk}{3}},     \label{eqn:(1-x)(1-y)}     \\
		x(1-y) + (1-x)y &= \frac{m(m-1) \binom{k}{2} \binom{k}{1}}{\binom{mk}{3}}, \label{eqn:redun}\\
		xy &= \frac{\binom{m}{3} k^3}{\binom{mk}{3}}. \label{eqn:xy}
	\end{align}
	This reduces to a quadratic and can be solved explicitly. Since $1 > xy > 0$ and $1 > (1-x)(1-y) > 0$, we necessarily have $0< x, y < 1$ for any solution and $x$ and $y$ will be valid choices for a probability. Since both sides sum to 1, one of the conditions is redundant, say (\ref{eqn:redun}). Rearranging (\ref{eqn:xy}) gives 
	\[ y = \frac{k^2 (m-2)(m-1)}{(km-2)(km-1)x}. \]
	Substituting this into (\ref{eqn:(1-x)(1-y)}) and rearranging shows
	\[(km-1)(km-2)x^2 + k (3 k m + m - k - 3)x  -  k^2(m-1)(m-2) = 0.\]
	The determinant of this quadratic is \[(k-1)k^2(m-1)(5km - k - m - 7)\] which is positive for $k, m \geq 3$. Hence, there are always two roots and there is a suitable choice for $x$. In fact, $y$ satisfies the same equation and will be the other root.
	
	
	In stage 3, we use $U_3(k)$ transpositions to shuffle the positions $1, 2, 3$ across the positions $1, \dots, k$. We use $k-1$ transpositions to shuffle the position $k+1$ into the positions $k +1, \dots, 2k$ and another $k-1$ transpositions to shuffle the position $2k +1$ into $2k + 1, \dots, 3k$. Since the counters can only be in positions $1, \dots, 3k$, stage 4 can be completed using $k$ copies of a $(3,m)$-shuffle. Explicitly, shuffle $i, k +i$ and $2k +i$ into  $i, k + i, \dots, (m-1)k + i$ for every $1 \leq i \leq k$.  
\end{proof}

We can also use Construction \ref{constr:main} to construct efficient shuffles when $t = 4$, but only in the special case $m = 3$. When $m = 3$, there are only four different unordered multisets for the $W_i$, namely $\{4, 0, 0\}$, $\{3, 1, 0\}$, $\{2, 2, 0\}$ and $\{2, 1, 1\}$. In particular,  we avoid the multiset of $\{1, 1, 1, 1, 0 \dots 0\}$ which appears to clash with $\{2, 2, 0, \dots, 0\}$ when using the simplest constructions for stage 2. We give the details below.

\begin{figure}
	\centering
	\begin{tikzpicture}[xscale=\textwidth/12cm]
	\draw (0, 0) -- (12, 0);
	\draw (0, 0.36363636363636365) -- (12, 0.36363636363636365);
	\draw (0, 0.7272727272727273) -- (12, 0.7272727272727273);
	\draw (0, 1.0909090909090908) -- (12, 1.0909090909090908);
	\draw (0, 1.4545454545454546) -- (12, 1.4545454545454546);
	\draw (0, 1.8181818181818183) -- (12, 1.8181818181818183);
	\draw (0, 2.1818181818181817) -- (12, 2.1818181818181817);
	\draw (0, 2.5454545454545454) -- (12, 2.5454545454545454);
	\draw (0, 2.909090909090909) -- (12, 2.909090909090909);
	\draw (0, 3.272727272727273) -- (12, 3.272727272727273);
	\draw (0, 3.6363636363636367) -- (12, 3.6363636363636367);
	\draw (0, 4) -- (12, 4);

	\draw (0.35294117647058826, 4) -- (0.35294117647058826, 3.6363636363636367);
	\draw node[style=vertex] at (0.35294117647058826, 4) {};
	\draw node[style=vertex] at (0.35294117647058826, 3.6363636363636367) {};
	\draw (0.7058823529411765, 4) -- (0.7058823529411765, 3.272727272727273);
	\draw node[style=vertex] at (0.7058823529411765, 4) {};
	\draw node[style=vertex] at (0.7058823529411765, 3.272727272727273) {};
	\draw (1.0588235294117647, 4) -- (1.0588235294117647, 2.909090909090909);
	\draw node[style=vertex] at (1.0588235294117647, 4) {};
	\draw node[style=vertex] at (1.0588235294117647, 2.909090909090909) {};
	\draw (1.411764705882353, 4) -- (1.411764705882353, 3.6363636363636367);
	\draw node[style=vertex] at (1.411764705882353, 4) {};
	\draw node[style=vertex] at (1.411764705882353, 3.6363636363636367) {};
	\draw (1.7647058823529413, 4) -- (1.7647058823529413, 3.272727272727273);
	\draw node[style=vertex] at (1.7647058823529413, 4) {};
	\draw node[style=vertex] at (1.7647058823529413, 3.272727272727273) {};
	\draw (2.1176470588235294, 4) -- (2.1176470588235294, 3.6363636363636367);
	\draw node[style=vertex] at (2.1176470588235294, 4) {};
	\draw node[style=vertex] at (2.1176470588235294, 3.6363636363636367) {};
	\draw (2.823529411764706, 4) -- (2.823529411764706, 2.5454545454545454);
	\draw node[style=vertex] at (2.823529411764706, 4) {};
	\draw node[style=vertex] at (2.823529411764706, 2.5454545454545454) {};
	\draw (3.1764705882352944, 3.6363636363636367) -- (3.1764705882352944, 2.1818181818181817);
	\draw node[style=vertex] at (3.1764705882352944, 3.6363636363636367) {};
	\draw node[style=vertex] at (3.1764705882352944, 2.1818181818181817) {};
	\draw (3.5294117647058827, 3.272727272727273) -- (3.5294117647058827, 1.0909090909090908);
	\draw node[style=vertex] at (3.5294117647058827, 3.272727272727273) {};
	\draw node[style=vertex] at (3.5294117647058827, 1.0909090909090908) {};
	\draw (4.235294117647059, 4) -- (4.235294117647059, 3.6363636363636367);
	\draw node[style=vertex] at (4.235294117647059, 4) {};
	\draw node[style=vertex] at (4.235294117647059, 3.6363636363636367) {};
	\draw (4.235294117647059, 2.5454545454545454) -- (4.235294117647059, 2.1818181818181817);
	\draw node[style=vertex] at (4.235294117647059, 2.5454545454545454) {};
	\draw node[style=vertex] at (4.235294117647059, 2.1818181818181817) {};
	\draw (4.235294117647059, 1.0909090909090908) -- (4.235294117647059, 0.7272727272727273);
	\draw node[style=vertex] at (4.235294117647059, 1.0909090909090908) {};
	\draw node[style=vertex] at (4.235294117647059, 0.7272727272727273) {};
	\draw (4.588235294117648, 4) -- (4.588235294117648, 3.272727272727273);
	\draw node[style=vertex] at (4.588235294117648, 4) {};
	\draw node[style=vertex] at (4.588235294117648, 3.272727272727273) {};
	\draw (4.588235294117648, 2.5454545454545454) -- (4.588235294117648, 1.8181818181818183);
	\draw node[style=vertex] at (4.588235294117648, 2.5454545454545454) {};
	\draw node[style=vertex] at (4.588235294117648, 1.8181818181818183) {};
	\draw (4.588235294117648, 1.0909090909090908) -- (4.588235294117648, 0.36363636363636365);
	\draw node[style=vertex] at (4.588235294117648, 1.0909090909090908) {};
	\draw node[style=vertex] at (4.588235294117648, 0.36363636363636365) {};
	\draw (4.9411764705882355, 4) -- (4.9411764705882355, 2.909090909090909);
	\draw node[style=vertex] at (4.9411764705882355, 4) {};
	\draw node[style=vertex] at (4.9411764705882355, 2.909090909090909) {};
	\draw (4.9411764705882355, 2.5454545454545454) -- (4.9411764705882355, 2.1818181818181817);
	\draw node[style=vertex] at (4.9411764705882355, 2.5454545454545454) {};
	\draw node[style=vertex] at (4.9411764705882355, 2.1818181818181817) {};
	\draw (4.9411764705882355, 1.0909090909090908) -- (4.9411764705882355, 0);
	\draw node[style=vertex] at (4.9411764705882355, 1.0909090909090908) {};
	\draw node[style=vertex] at (4.9411764705882355, 0) {};
	\draw (5.294117647058824, 4) -- (5.294117647058824, 3.6363636363636367);
	\draw node[style=vertex] at (5.294117647058824, 4) {};
	\draw node[style=vertex] at (5.294117647058824, 3.6363636363636367) {};
	\draw (5.294117647058824, 2.5454545454545454) -- (5.294117647058824, 1.4545454545454546);
	\draw node[style=vertex] at (5.294117647058824, 2.5454545454545454) {};
	\draw node[style=vertex] at (5.294117647058824, 1.4545454545454546) {};
	\draw (5.647058823529412, 4) -- (5.647058823529412, 3.272727272727273);
	\draw node[style=vertex] at (5.647058823529412, 4) {};
	\draw node[style=vertex] at (5.647058823529412, 3.272727272727273) {};
	\draw (5.647058823529412, 2.5454545454545454) -- (5.647058823529412, 2.1818181818181817);
	\draw node[style=vertex] at (5.647058823529412, 2.5454545454545454) {};
	\draw node[style=vertex] at (5.647058823529412, 2.1818181818181817) {};
	\draw (6, 4) -- (6, 3.6363636363636367);
	\draw node[style=vertex] at (6, 4) {};
	\draw node[style=vertex] at (6, 3.6363636363636367) {};
	\draw (6.705882352941177, 4) -- (6.705882352941177, 2.5454545454545454);
	\draw node[style=vertex] at (6.705882352941177, 4) {};
	\draw node[style=vertex] at (6.705882352941177, 2.5454545454545454) {};
	\draw (7.058823529411765, 4) -- (7.058823529411765, 1.0909090909090908);
	\draw node[style=vertex] at (7.058823529411765, 4) {};
	\draw node[style=vertex] at (7.058823529411765, 1.0909090909090908) {};
	\draw (7.411764705882353, 4) -- (7.411764705882353, 2.5454545454545454);
	\draw node[style=vertex] at (7.411764705882353, 4) {};
	\draw node[style=vertex] at (7.411764705882353, 2.5454545454545454) {};
	\draw (8.11764705882353, 3.6363636363636367) -- (8.11764705882353, 2.1818181818181817);
	\draw node[style=vertex] at (8.11764705882353, 3.6363636363636367) {};
	\draw node[style=vertex] at (8.11764705882353, 2.1818181818181817) {};
	\draw (8.470588235294118, 3.6363636363636367) -- (8.470588235294118, 0.7272727272727273);
	\draw node[style=vertex] at (8.470588235294118, 3.6363636363636367) {};
	\draw node[style=vertex] at (8.470588235294118, 0.7272727272727273) {};
	\draw (8.823529411764707, 3.6363636363636367) -- (8.823529411764707, 2.1818181818181817);
	\draw node[style=vertex] at (8.823529411764707, 3.6363636363636367) {};
	\draw node[style=vertex] at (8.823529411764707, 2.1818181818181817) {};
	\draw (9.529411764705882, 3.272727272727273) -- (9.529411764705882, 1.8181818181818183);
	\draw node[style=vertex] at (9.529411764705882, 3.272727272727273) {};
	\draw node[style=vertex] at (9.529411764705882, 1.8181818181818183) {};
	\draw (9.882352941176471, 3.272727272727273) -- (9.882352941176471, 0.36363636363636365);
	\draw node[style=vertex] at (9.882352941176471, 3.272727272727273) {};
	\draw node[style=vertex] at (9.882352941176471, 0.36363636363636365) {};
	\draw (10.23529411764706, 3.272727272727273) -- (10.23529411764706, 1.8181818181818183);
	\draw node[style=vertex] at (10.23529411764706, 3.272727272727273) {};
	\draw node[style=vertex] at (10.23529411764706, 1.8181818181818183) {};
	\draw (10.941176470588236, 2.909090909090909) -- (10.941176470588236, 1.4545454545454546);
	\draw node[style=vertex] at (10.941176470588236, 2.909090909090909) {};
	\draw node[style=vertex] at (10.941176470588236, 1.4545454545454546) {};
	\draw (11.294117647058824, 2.909090909090909) -- (11.294117647058824, 0);
	\draw node[style=vertex] at (11.294117647058824, 2.909090909090909) {};
	\draw node[style=vertex] at (11.294117647058824, 0) {};
	\draw (11.647058823529413, 2.909090909090909) -- (11.647058823529413, 1.4545454545454546);
	\draw node[style=vertex] at (11.647058823529413, 2.909090909090909) {};
	\draw node[style=vertex] at (11.647058823529413, 1.4545454545454546) {};
\end{tikzpicture}
	\caption{\label{fig:4-12}A $(4,12)$-shuffle using just 35 transpositions constructed as in the proof of Lemma \ref{lem:4}.}
\end{figure}
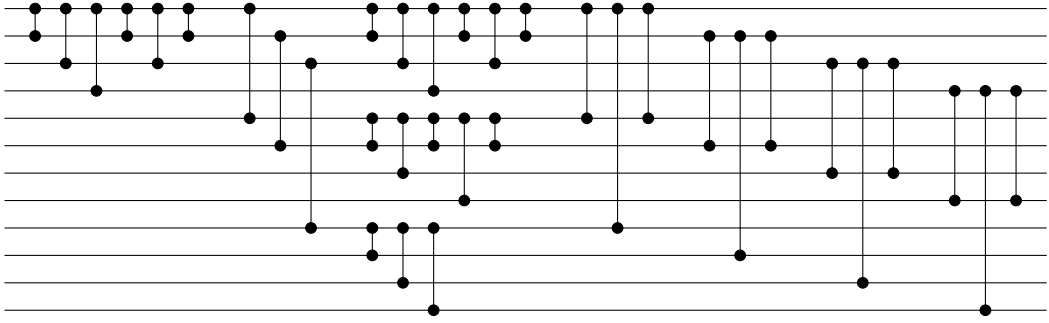

\begin{lemma}
	\label{lem:4}
		Let $k \geq 4$. Then 
		\[U_4(3k) \leq 4k + 8 + U_4(k) + U_2(k).\]
	\end{lemma}
	\begin{proof}
		We apply Construction \ref{constr:main} when $t = 4$  and $m = 3$. Start by shuffling the counters using $6$ transpositions, before using the transpositions $(1, k +1, x)$, $(2, k + 2, y)$ and $(3, 2k +1, z)$. Condition (\ref{eqn-cond}) becomes
		\begin{align*}
			\frac{3 \binom{k}{4}}{\binom{3k}{4}} &= (1-x)(1-y)(1-z), \\
			\frac{6 k \binom{k}{3}}{\binom{3k}{4}} &= x(1-y)(1-z) + y(1-x)(1-z) + z(1-x)(1-y),\\
			\frac{3 \binom{k}{2}^2}{\binom{3k}{4}} &= xy (1-z),\\
			\frac{3 k^2 \binom{k}{2} }{\binom{3k}{4}} &= x(1-y)z + (1-x)yz + xyz.
		\end{align*}
		which has solutions for $k \geq 4$. This can be seen by substituting in the value 
		\[z = \frac{19k^2 - 11k - k \sqrt{37k^2 - 94k + 49}}{27k^2 - 27k + 6}\] and proceeding as in the proof of Lemma \ref{lem:3}.
	
	
	In stage 4 shuffle the first four positions into the first $k$ positions using $U_4(k)$ transpositions. Then shuffle $k + 1$ and $k+2$ into $k+1, \dots, 2k$ using $U_2(k)$ transpositions and $2k + 1$ into $2k+1, \dots, 3k$ using $k-1$ transpositions.  Finally, end by shuffling the positions $i, k+ i, 2k + i$ for every $1 \leq i \leq k$.
\end{proof}

We end this section by proving Theorem \ref{thm:kn} using the constructions from Section \ref{sec:2parts} and Section \ref{sec:preliminaries} as well as Lemma \ref{lem:3}.

\begin{proof}[Proof of  Theorem \ref{thm:kn}]
	We use the simple bound $U_3(3k) \leq 8k -3$ for $k \geq 3$ which follows from combining the trivial bound $U_3(k) \leq 3k -6$ with Lemma \ref{lem:3}.
	Lemma \ref{lem:simple-divide} gives 
	\begin{align*}
		U_{3\ell} (3k) &\leq U_3(3k - 3(\ell-1)) + U_{3(\ell-1)}(3k)\\
		&\leq 8 (k - \ell + 1) -3 + U_{3(\ell-1)}(3k).
	\end{align*}
	In particular, if we denote $U_{3\ell}(3k) = 9k\ell - \binom{3\ell + 1}{2} - a_\ell$, we get the recursion relation
	\[a_\ell \geq (k - \ell +1) - 3 + a_{\ell -1},\]
	with the initial condition $a_1 = k - 3$. Hence, \[a_\ell \geq \frac{1}{2} \ell \left( 2k - \ell - 5\right). \]
	Now let $n = 3k +r$ for $0 \leq r < 3$, let $t = 3\ell +s$ for $0 \leq s <3$ and suppose $n \geq 9$. Using Lemma \ref{lem:sweeping} and Lemma \ref{lem:simple-divide}, we have 
	\[U_t(n) \leq tn - \binom{t+1}{2} - \frac{1}{2} \ell \left( 2k - \ell - 5\right).\]
	In particular, $U_t(n) \leq (1-\delta) \left( tn - \binom{t+1}{2} \right) $ for all
	\[\delta \leq \frac{\ell \left( 2k - \ell - 5\right)}{2 \left( (3\ell +2) (3k +2) - \binom{3\ell +2}{2}\right)}.\]
	This is an increasing function of $k$ so the worst case is when $k = \ell$ where it takes the value \[\frac{(\ell - 5) \ell}{9 \ell^2 + 9\ell +2}.\]
	This is greater than $3/190$ for all $\ell \geq 6$. The cases where $\ell \leq 5$ follow from the second part of the theorem. Indeed, the second part of the theorem guarantees the existence of a constant $\delta_t$ such that $U_t(n) \leq (1- \delta_t)\left( tn - \binom{t+1}{2}\right)$ for large enough $n$, and we already know that we can beat the bound by an additive factor for all $t \geq 3$ and $n \geq 6$. 
	
    The second part of the theorem follows almost immediately from the proof of Theorem \ref{thm:wrongconstant} where we proved the following bound.
	\[U_{3\ell}(n) \leq 2\ell n + 5 \ell \log_2 n + 6 \ell - 3\ell^2.\]
	If $t = 3 \ell + r$ for some $0 \leq r < 2$, we have 
	\begin{align*}
	    U_{3 \ell +r}(n) &\leq U_r(n - 3 \ell) + U_{3\ell}(n)\\
	    &\leq r(n- 3 \ell) - \binom{r+1}{2} + U_{3\ell}(n)\\
	    &\leq (2\ell + r)n + 5 \ell \log_2 n + (6 - 3r) \ell - 3 \ell^2 - \binom{r+1}{2}\\
	    &= (2\ell + r)n + O_\ell \left( \log n\right).
	\end{align*}
\end{proof}

\section{Star transposition shuffles}
\label{sec:improved-star}

We now consider the problem of generating efficient shuffles which only use star transpositions. 
Although star transpositions are much more restrictive, any transposition shuffle can be converted to a shuffle which only uses star transpositions by replacing the lazy transposition $(a, b, p)$ with the transpositions $(1, a, 1)(1,b,p) (1,a,1)$. This means that all the constructions above can be converted to ones using only star transpositions, but it turns out only the construction from Lemma \ref{lem:4} leads to an efficient shuffle. All of the efficient constructions we have given make use of several ``subshuffles" which shuffle certain positions together, and these can be taken to use transpositions of the form $(x, \cdot)$ for some suitable value of $x$. We can ``switch" the central position of the star to position $x$ using the transposition $(1,x, 1)$, use the subshuffle with $x$ replaced by 1 and then use $(1,x,1)$ again to switch the central position back to $1$. If we carefully keep track of which positions have been switched, we can often do better and avoid switching the central position back to 1 the majority of the time. Although the construction in Lemma \ref{lem:4} is not the most efficient for general transpositions, it uses relatively few ``subshuffles" and we only need to switch the central position relatively few times.

Let $\widetilde{U}_{t}(n)$ be the minimum number of lazy transpositions in a $(t, n)$-shuffle where all transpositions are star transpositions. Lemma \ref{lem:sweeping} shows that the trivial upper bound of $tn - \binom{t+1}{2}$ holds for $\widetilde{U}_t(n)$, but the following theorem shows it is again possible to beat the upper bound by a constant factor when $t = 4$ and $n$ is large enough. By repeatedly shuffling in 4 counters at a time, we can construct efficient shuffles for all $t \geq 4$.

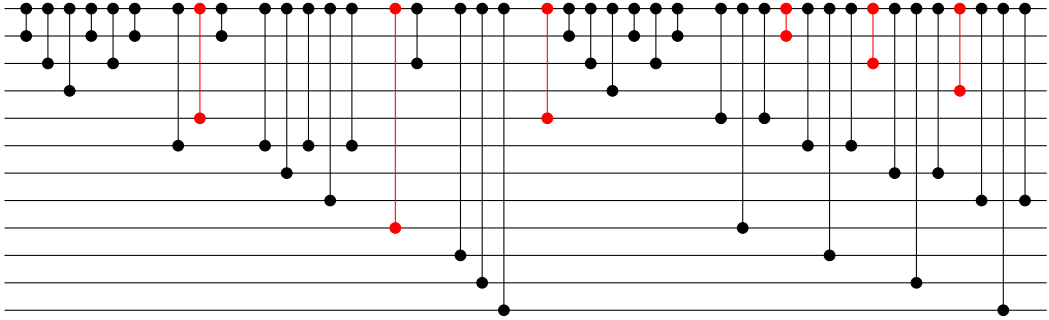
\begin{figure}
    \centering
    \begin{tikzpicture}[xscale=\textwidth/12cm]
	\draw (0, 0) -- (12, 0);
	\draw (0, 0.36363636363636365) -- (12, 0.36363636363636365);
	\draw (0, 0.7272727272727273) -- (12, 0.7272727272727273);
	\draw (0, 1.0909090909090908) -- (12, 1.0909090909090908);
	\draw (0, 1.4545454545454546) -- (12, 1.4545454545454546);
	\draw (0, 1.8181818181818183) -- (12, 1.8181818181818183);
	\draw (0, 2.1818181818181817) -- (12, 2.1818181818181817);
	\draw (0, 2.5454545454545454) -- (12, 2.5454545454545454);
	\draw (0, 2.909090909090909) -- (12, 2.909090909090909);
	\draw (0, 3.272727272727273) -- (12, 3.272727272727273);
	\draw (0, 3.6363636363636367) -- (12, 3.6363636363636367);
	\draw (0, 4) -- (12, 4);

	\draw (0.25, 4) -- (0.25, 3.6363636363636367);
	\draw node[style=vertex] at (0.25, 4) {};
	\draw node[style=vertex] at (0.25, 3.6363636363636367) {};
	\draw (0.5, 4) -- (0.5, 3.272727272727273);
	\draw node[style=vertex] at (0.5, 4) {};
	\draw node[style=vertex] at (0.5, 3.272727272727273) {};
	\draw (0.75, 4) -- (0.75, 2.909090909090909);
	\draw node[style=vertex] at (0.75, 4) {};
	\draw node[style=vertex] at (0.75, 2.909090909090909) {};
	\draw (1, 4) -- (1, 3.6363636363636367);
	\draw node[style=vertex] at (1, 4) {};
	\draw node[style=vertex] at (1, 3.6363636363636367) {};
	\draw (1.25, 4) -- (1.25, 3.272727272727273);
	\draw node[style=vertex] at (1.25, 4) {};
	\draw node[style=vertex] at (1.25, 3.272727272727273) {};
	\draw (1.5, 4) -- (1.5, 3.6363636363636367);
	\draw node[style=vertex] at (1.5, 4) {};
	\draw node[style=vertex] at (1.5, 3.6363636363636367) {};
	\draw (2, 4) -- (2, 2.1818181818181817);
	\draw node[style=vertex] at (2, 4) {};
	\draw node[style=vertex] at (2, 2.1818181818181817) {};
	\draw[red] (2.25, 4) -- (2.25, 2.5454545454545454);
	\draw[red] node[style=vertex] at (2.25, 4) {};
	\draw[red] node[style=vertex] at (2.25, 2.5454545454545454) {};
	\draw (2.5, 4) -- (2.5, 3.6363636363636367);
	\draw node[style=vertex] at (2.5, 4) {};
	\draw node[style=vertex] at (2.5, 3.6363636363636367) {};
	\draw (3, 4) -- (3, 2.1818181818181817);
	\draw node[style=vertex] at (3, 4) {};
	\draw node[style=vertex] at (3, 2.1818181818181817) {};
	\draw (3.25, 4) -- (3.25, 1.8181818181818183);
	\draw node[style=vertex] at (3.25, 4) {};
	\draw node[style=vertex] at (3.25, 1.8181818181818183) {};
	\draw (3.5, 4) -- (3.5, 2.1818181818181817);
	\draw node[style=vertex] at (3.5, 4) {};
	\draw node[style=vertex] at (3.5, 2.1818181818181817) {};
	\draw (3.75, 4) -- (3.75, 1.4545454545454546);
	\draw node[style=vertex] at (3.75, 4) {};
	\draw node[style=vertex] at (3.75, 1.4545454545454546) {};
	\draw (4, 4) -- (4, 2.1818181818181817);
	\draw node[style=vertex] at (4, 4) {};
	\draw node[style=vertex] at (4, 2.1818181818181817) {};
	\draw[red] (4.5, 4) -- (4.5, 1.0909090909090908);
	\draw[red] node[style=vertex] at (4.5, 4) {};
	\draw[red] node[style=vertex] at (4.5, 1.0909090909090908) {};
	\draw (4.75, 4) -- (4.75, 3.272727272727273);
	\draw node[style=vertex] at (4.75, 4) {};
	\draw node[style=vertex] at (4.75, 3.272727272727273) {};
	\draw (5.25, 4) -- (5.25, 0.7272727272727273);
	\draw node[style=vertex] at (5.25, 4) {};
	\draw node[style=vertex] at (5.25, 0.7272727272727273) {};
	\draw (5.5, 4) -- (5.5, 0.36363636363636365);
	\draw node[style=vertex] at (5.5, 4) {};
	\draw node[style=vertex] at (5.5, 0.36363636363636365) {};
	\draw (5.75, 4) -- (5.75, 0);
	\draw node[style=vertex] at (5.75, 4) {};
	\draw node[style=vertex] at (5.75, 0) {};
	\draw[red] (6.25, 4) -- (6.25, 2.5454545454545454);
	\draw[red] node[style=vertex] at (6.25, 4) {};
	\draw[red] node[style=vertex] at (6.25, 2.5454545454545454) {};
	\draw (6.5, 4) -- (6.5, 3.6363636363636367);
	\draw node[style=vertex] at (6.5, 4) {};
	\draw node[style=vertex] at (6.5, 3.6363636363636367) {};
	\draw (6.75, 4) -- (6.75, 3.272727272727273);
	\draw node[style=vertex] at (6.75, 4) {};
	\draw node[style=vertex] at (6.75, 3.272727272727273) {};
	\draw (7, 4) -- (7, 2.909090909090909);
	\draw node[style=vertex] at (7, 4) {};
	\draw node[style=vertex] at (7, 2.909090909090909) {};
	\draw (7.25, 4) -- (7.25, 3.6363636363636367);
	\draw node[style=vertex] at (7.25, 4) {};
	\draw node[style=vertex] at (7.25, 3.6363636363636367) {};
	\draw (7.5, 4) -- (7.5, 3.272727272727273);
	\draw node[style=vertex] at (7.5, 4) {};
	\draw node[style=vertex] at (7.5, 3.272727272727273) {};
	\draw (7.75, 4) -- (7.75, 3.6363636363636367);
	\draw node[style=vertex] at (7.75, 4) {};
	\draw node[style=vertex] at (7.75, 3.6363636363636367) {};
	\draw (8.25, 4) -- (8.25, 2.5454545454545454);
	\draw node[style=vertex] at (8.25, 4) {};
	\draw node[style=vertex] at (8.25, 2.5454545454545454) {};
	\draw (8.5, 4) -- (8.5, 1.0909090909090908);
	\draw node[style=vertex] at (8.5, 4) {};
	\draw node[style=vertex] at (8.5, 1.0909090909090908) {};
	\draw (8.75, 4) -- (8.75, 2.5454545454545454);
	\draw node[style=vertex] at (8.75, 4) {};
	\draw node[style=vertex] at (8.75, 2.5454545454545454) {};
	\draw[red] (9, 4) -- (9, 3.6363636363636367);
	\draw[red] node[style=vertex] at (9, 4) {};
	\draw[red] node[style=vertex] at (9, 3.6363636363636367) {};
	\draw (9.25, 4) -- (9.25, 2.1818181818181817);
	\draw node[style=vertex] at (9.25, 4) {};
	\draw node[style=vertex] at (9.25, 2.1818181818181817) {};
	\draw (9.5, 4) -- (9.5, 0.7272727272727273);
	\draw node[style=vertex] at (9.5, 4) {};
	\draw node[style=vertex] at (9.5, 0.7272727272727273) {};
	\draw (9.75, 4) -- (9.75, 2.1818181818181817);
	\draw node[style=vertex] at (9.75, 4) {};
	\draw node[style=vertex] at (9.75, 2.1818181818181817) {};
	\draw[red] (10, 4) -- (10, 3.272727272727273);
	\draw[red] node[style=vertex] at (10, 4) {};
	\draw[red] node[style=vertex] at (10, 3.272727272727273) {};
	\draw (10.25, 4) -- (10.25, 1.8181818181818183);
	\draw node[style=vertex] at (10.25, 4) {};
	\draw node[style=vertex] at (10.25, 1.8181818181818183) {};
	\draw (10.5, 4) -- (10.5, 0.36363636363636365);
	\draw node[style=vertex] at (10.5, 4) {};
	\draw node[style=vertex] at (10.5, 0.36363636363636365) {};
	\draw (10.75, 4) -- (10.75, 1.8181818181818183);
	\draw node[style=vertex] at (10.75, 4) {};
	\draw node[style=vertex] at (10.75, 1.8181818181818183) {};
	\draw[red] (11, 4) -- (11, 2.909090909090909);
	\draw[red] node[style=vertex] at (11, 4) {};
	\draw[red] node[style=vertex] at (11, 2.909090909090909) {};
	\draw (11.25, 4) -- (11.25, 1.4545454545454546);
	\draw node[style=vertex] at (11.25, 4) {};
	\draw node[style=vertex] at (11.25, 1.4545454545454546) {};
	\draw (11.5, 4) -- (11.5, 0);
	\draw node[style=vertex] at (11.5, 4) {};
	\draw node[style=vertex] at (11.5, 0) {};
	\draw (11.75, 4) -- (11.75, 1.4545454545454546);
	\draw node[style=vertex] at (11.75, 4) {};
	\draw node[style=vertex] at (11.75, 1.4545454545454546) {};
\end{tikzpicture}
    \caption[A $(4,12)$-shuffle using only star transpositions constructed as in the proof of Lemma \ref{lem:efficient-star}.]{A $(4,12)$-shuffle constructed as in the proof of Lemma \ref{lem:efficient-star}. The red transpositions indicate the additional transpositions inserted to ``switch" the central position and ensure all transpositions are of the form $(1, \cdot)$. These fire with probability 1. Note that for this small value of $k$, this shuffle is worse than the trivial upper bound.}
    \label{fig:star-4-12}
\end{figure}
\begin{lemma}
    \label{lem:efficient-star}
    For every $k \geq 3$, we have 
    \begin{equation*}
        \widetilde{U}_4(3k) \leq 5k + 10 + \widetilde{U}_2(k) + \widetilde{U}_4(k).
    \end{equation*}
    In particular, 
    \begin{equation*}
        \widetilde{U}_4(3k) \leq 11k -3.
    \end{equation*}
\end{lemma}
\begin{proof}
	We follow the construction given in the proof of Lemma \ref{lem:4} but with some additional transpositions which always fire. Start by shuffling the 4 counters using 6 star transpositions. In order to minimise the number of times we use deterministic transpositions to ``switch" the central position, we will interleave parts of stages 2 and 3. Begin with the transpositions $(1,k + 2, x)$, $(1, k + 1, 1)$ and $(1, 2, y)$ where $x$ and $y$ are as given in the proof of Lemma \ref{lem:4}. Shuffle the positions $1$ and $k+2$ over the positions $1, k +2, k + 3, \dots, 2k$ using $\widetilde{U}_2(k)$ star transpositions. Next,  use the transposition $(1, 2k +1, 1)$ followed by $(1, 3, z)$ where $z$ is also as in the proof of Lemma \ref{lem:4}. Shuffle the position $1$ over $1, 2k + 2, \dots, 3k$ using $k-1$ star transpositions. To finish stages 2 and 3, use the transposition $(1, k +1, 1)$ to return the centre to its original position and shuffle the positions $1, 2, 3, 4$ over the positions $1, 2, \dots, k$. Note that while the centre is back in its original positions, the positions $k+1$ and $2k+1$ have ``swapped", but this does not matter for stage 4.
    
    Begin stage 4 by shuffling the positions $1, k +1$ and $2k + 1$. For each $2 \leq i \leq k$, use the transposition $(1, i, 1)$ and then the transpositions $(1, k + i)$, $(1, 2k + i)$ and $(1, k + i)$ to shuffle the positions $1, k + i$ and $2k + i$. 
\end{proof}

To prove Theorem \ref{thm:improved-star}, we can mimic the proofs used for general transpositions for which we only need versions of Lemma \ref{lem:sweeping} and Lemma \ref{lem:simple-divide}. The former extends immediately to give 
\[\widetilde{U}_t(n) \leq t + \widetilde{U}_t(n-1),\]
but the latter does not immediately extend. Instead, we need to add a single transposition to switch the centre, and we get the bound
\[\widetilde{U}_t(n) \leq \widetilde{U}_{t - m}(n - m) + \widetilde{U}_m(n) + 1.\]
The proof of Theorem \ref{thm:improved-star} is now very similar to the proofs of Theorem \ref{thm:wrongconstant} and Theorem \ref{thm:kn} given in the general case, and is omitted.

\section{Open problems}
\label{sec:discussion}
While we have given constructions which show that the upper bound for $U(n)$ can be improved by a constant factor, the new upper bounds are still a long way from the best known lower bound. The biggest open problem is to determine the correct asymptotic for $U(n)$, and we conjecture that $U(n) = o(n^2)$.
\begin{problem}
	What is the correct asymptotic for $U(n)$? Is it true that $U(n) = o(n^2)$?
\end{problem}
A natural candidate for the correct asymptotic behaviour is $U(n) = \Theta(n \log n)$ and it would already be very interesting to establish that this is not the case (i.e. to show that $U(n) = \omega(n \log n)$).

We also considered constructing transposition shuffles where all transpositions are of the form $(1, \cdot)$. 
The minimum number of transpositions in such a shuffle is between $U(n)$ and $3 U(n)$, but the restrictive nature of the transpositions might make this variant of the problem more tractable.

\begin{problem}
	What is the minimum number of lazy transpositions in a transposition shuffle if all transpositions are of the form $(1,\cdot)$?
\end{problem}

We introduced the study of $(t,n)$-shuffles and showed that the trivial upper bound can be improved by a constant factor for all $t \geq 3 $ and $n \geq 6$. However, we have been unable to find any $(2,n)$-shuffles using fewer than $2n - 3$ transpositions, and in \cite{reachability} we conjectured that this is optimal:
\begin{conjecture}
	For $n \geq 2$,  $U_2(n) = 2n -3$.
\end{conjecture}
We remark that, if this were true, it would match nicely with selection networks (see \cite{knuth1997art} for a discussion of selection and sorting networks). As a first case, it would be interesting to close the gap when restricting to star transpositions and confirm the conjecture in this case.

\bibliographystyle{abbrv-bold}
\bibliography{transposition-shuffles}

\appendix
\section{Best known upper bounds}
\begin{table}
    \centering
\begin{tabular}{rrrr} \toprule
$n$ & $\binom{n}{2}$ & \thead{Upper\\bound} & \thead{Proportion \\of $\binom{n}{2}$}\\\midrule
2 & 1 & 1 & 1.000\\
3 & 3 & 3 & 1.000\\
4 & 6 & 6 & 1.000\\
5 & 10 & 10 & 1.000\\
6 & 15 & 14 & 0.933\\
7 & 21 & 20 & 0.952\\
8 & 28 & 26 & 0.929\\
9 & 36 & 33 & 0.917\\
10 & 45 & 41 & 0.911\\
11 & 55 & 50 & 0.909\\
12 & 66 & 58 & 0.879\\
13 & 78 & 69 & 0.885\\
14 & 91 & 80 & 0.879\\
15 & 105 & 91 & 0.867\\
16 & 120 & 103 & 0.858\\
17 & 136 & 117 & 0.860\\
18 & 153 & 130 & 0.850\\
19 & 171 & 145 & 0.848\\
20 & 190 & 160 & 0.842\\
21 & 210 & 176 & 0.838\\
22 & 231 & 193 & 0.835\\
23 & 253 & 211 & 0.834\\
24 & 276 & 227 & 0.822\\
25 & 300 & 247 & 0.823\\
26 & 325 & 267 & 0.822\\
27 & 351 & 286 & 0.815\\
28 & 378 & 307 & 0.812\\
29 & 406 & 330 & 0.813\\
30 & 435 & 351 & 0.807\\
31 & 465 & 375 & 0.806\\
32 & 496 & 398 & 0.802\\
\bottomrule
\end{tabular}
\caption[The best known upper bounds for $U(n)$.]{The best known upper bounds for $U(n)$. The upper bound is known to be correct for $n = 2, 3, 4$ and computer experiments suggest it holds for $n = 5$ as well \cite{angel2018perfect}.}
    \label{tab:my_label}
\end{table}
\end{document}